\documentclass{cimart}

\usepackage{enumerate}

\DeclareMathOperator{\Aut}{Aut}

\title[Analogs of the lower and upper central series in skew braces: a survey]{Analogs of the lower and upper central series \\in skew braces: a survey}

\author{Cindy (Sin Yi) Tsang}

\authorinfo{Ochanomizu University, Japan}{tsang.sin.yi@ocha.ac.jp}

\abstract{%
    A skew brace is a ring-like and group-like algebraic structure that was introduced in the study of set-theoretic solutions to the Yang-Baxter equation. In this survey paper, we shall consider the left series, right series, socle series, and annihilator series of skew braces. They may be regarded as analogs of the lower and upper central series of groups. Other than some well-known facts regarding these series, we shall prove several new results about the relationship among their terms. We shall also consider the lower central series of skew braces that was defined by Bonatto and Jedli\v{c}ka. As we shall explain, it seems to be the ``correct" analog of the lower central series for skew braces. Concerning this, we shall also discuss the notion of the lower central series of ideals that is due to Ballester-Bolinches et al.
    }



\VOLUME{33}
\YEAR{2025}
\ISSUE{3}
\NUMBER{11}
\DOI{https://doi.org/10.46298/cm.15314}

\begin{document}

\section{Introduction}

Skew brace is an algebraic structure that was introduced in the study of set-theoretic solutions to the Yang-Baxter equation \cite{Rump0, GV}. We shall omit the details, but skew brace also has connections with other objects, such as regular subgroups of the holomorph \cite[Section 4]{GV}, Rota-Baxter groups \cite{RB}, post-groups \cite{post}, and Hopf-Galois structures \cite{SV, ST}. 

A \textit{skew (left) brace} is any set $A= (A,\cdot,\circ)$ endowed with two group operations $\cdot$ and $\circ$ such that the so-called (left) brace relation
\begin{equation}\label{brace relation} a\circ (b \cdot c ) = (a\circ b)\cdot a^{-1}\cdot (a\circ c)
\end{equation}
holds for all $a,b,c\in A$. For each $a\in A$, we shall write $a^{-1}$ for the inverse of $a$ in $(A,\cdot)$, and $\overline{a}$ for the inverse of $a$ in $(A,\circ)$. It is easy to show that $(A,\cdot)$ and $(A,\circ)$ share the same identity element, which we shall denote by $1$. We refer to $(A,\cdot)$ and $(A,\circ)$, respectively, as the \textit{additive group} and \textit{multiplicative group} of the skew brace $A$. A \textit{brace} is a skew brace with an abelian additive group.

For example, given any group $(A,\cdot)$, we can turn it into a skew brace by defining $\circ$ to be the same group operation $\cdot$, or the opposite operation $\cdot^{\mbox{\tiny op}}$ of $\cdot$ that is given by $a\cdot^{\mbox{\tiny op}} b = b\cdot a$. In some sense, skew braces that arise in this way are the same thing as groups. For this reason, skew braces of the forms $(A,\cdot,\cdot)$ and $(A,\cdot,\cdot^{\mbox{\tiny op}})$, respectively, are said to be \textit{trivial} and \textit{almost trivial}.

There are many similarities between groups and skew braces. Indeed, many concepts and results in group theory have been generalized to skew braces. We briefly survey a few examples here for the interested reader.
\begin{enumerate}[$\bullet$]
\item Factorization of skew braces was considered in \cite{factor} and an analog of It\^{o}'s theorem was obtained in \cite[Theorem 3.5]{factor}. Using opposite skew braces, the author \cite{TsangIto} proved a genuine generalization of It\^{o}'s theorem for skew braces, from which the usual It\^ {o}'s theorem may be recovered as a special case. Ballester-Bolinches announced at the conference ``Solving the Yang-Baxter Equation New Frontiers and Approaches" at Levico Terme, Italy (2025) that he and his collaborators obtained another such generalization. Their result (now available in the preprint \cite{abelian-product}) has the advantage that it does not involve the opposite skew brace.
\item Isoclinism of skew braces and the notion of stem skew braces were introduced in \cite{isoclinism}. As proven in \cite[Theorem 2.18]{isoclinism}, analogous to the case of groups, every skew brace is isoclinic to a stem skew brace. Schur cover of finite skew braces was defined in \cite{Schur}, and just like the case of groups, it is unique up to isoclinism by \cite[Theorem 3.19]{Schur}. 
\item Representation of skew braces was introduced in \cite{Schur} in relation to Schur covers, and its definition is based on \cite{Zhu}. Its theory was further explored in \cite{Kozakai, rep}. P\'{e}rez-Altarriba presented at the aforementioned conference at Levico Terme, Italy, that he and his collaborators came up with a different definition of representation of skew braces that they believe better reflects the structures of skew braces. It heavily relies on the connection between skew braces and trifactorized groups described in \cite{trifactor}.
\item Solvability of braces was defined in \cite{asymmetric} and was later extended to all skew braces in \cite{KSV}. Their definition has the pathology that every group, when regarded as a trivial skew brace, would be solvable. To avoid this issue, a new definition was introduced in \cite{soluble}, and it was also shown in \cite[Theorem B]{soluble} that solvable skew braces have a rich ideal structure.
\item Supersolvability of skew braces was studied in \cite{supersoluble}. Among the many similarities with supersolvable groups, it was shown in \cite[Corollary 3.9]{supersoluble} that skew braces of squarefree order are supersolvable and in \cite[Theorem 3.11]{supersoluble} that supersolvable skew braces always have Sylow towers.
\end{enumerate}
Finally, we mention two families of skew braces that seem to be much more difficult to classify than their group-theoretic counterparts.
\begin{enumerate}[$\bullet$]
\item A skew brace $A =(A,\cdot,\circ) $ is  \textit{simple} if $A\neq1$ and $A$ has no ideal except $1$ and itself. If either $(A,\cdot)$ or $(A,\circ)$ is a simple group, then certainly $A$ is a simple skew brace. However, there are infinitely many simple skew braces whose additive and multiplicative groups are both non-abelian but solvable \cite{Byott-simple}, as well as simple braces that are not of prime order \cite{simple}.
\item A skew brace $A=(A,\cdot,\circ)$ is \textit{Dedekind} if every sub-skew brace of $A$ is an ideal. This concept was introduced in \cite{Dedekind}, and as remarked there, it is a very challenging problem to try and obtain a structural theorem for such skew braces. Nonetheless, by \cite[Theorem~13]{Dedekind}, we know that every finite Dedekind brace is annihilator nilpotent.\end{enumerate}
We refer the reader to \cite{Jordan, min,radical, Smok2} for some further work that explores the similarities between groups (or rings) and skew braces.

In this paper, we shall focus on the nilpotency of skew braces. In the context of skew braces, there are many types of nilpotency, coming from the different analogs of central series. In Sections \ref{sec:lower} and \ref{sec:upper}, we shall survey the central series that are relevant in the definition of the following types of nilpotency:
\begin{enumerate}[(i)]
\itemsep=-4pt
\item left nilpotency,
\item right nilpotency,
\item socle nilpotency,
\item annihilator nilpotency.
\end{enumerate}
In Section \ref{sec:compare}, we shall survey some results that compare the different types of nilpotency. We give self-contained proofs for all of them, and as a consequence, we see that finiteness in \cite[Corollary 2.11]{central} can be dropped. In Section \ref{sec:relation}, we shall consider some inclusion relations among the terms of the various central series. We give counterexamples for all but one of them, and we prove that the remaining inclusion always holds (see Theorem~\ref{main thm}). Finally, in Section \ref{sec:central}, we shall delve into annihilator nilpotency. We discuss another analog of the lower central series, which seems to be the ``correct" analog and characterizes annihilator nilpotency. We also explore the notion of ``relative" annihilator nilpotency for ideals that is due to \cite{central-ideal}. We mention in passing that, related to left and right nilpotency, we have the concepts of left and right nil, which were first introduced by \cite{nilpotent} and later investigated in \cite{nil}.

\begin{remark}
Nilpotency of skew braces is not only important from an algebraic point of view but also from a set-theoretic solution perspective. As pointed out in \cite{nilpotent'}, nilpotency of skew braces is closely related to the multipermutation level of solutions (also see \cite[Proposition 5.3]{Jordan}). In addition, two classes of skew braces that behave like $FC$-groups were considered in \cite{derived}, and it was shown that they are closely connected to solutions whose derived solution is indecomposable. Nilpotency of such skew braces was also studied and some analogs of group-theoretic results were proven in \cite[Section 3]{derived}. 
\end{remark}

\begin{remark}
Skew brace is a powerful tool for analyzing set-theoretic solutions of the Yang-Baxter equation. However, all of the solutions (including the $z$-deformations defined in \cite{novel}) that arise from skew braces are non-degenerate. To study the solutions that are only left non-degenerate, a similar algebraic structure called \textit{semi-brace} was introduced in \cite{semi}, and it was shown in \cite[Theorem 9]{semi} that one can construct left non-degenerate solutions from semi-braces. The four types of nilpotency mentioned above can be defined in the setting of semi-braces, and one can find the details in \cite{semi-nil}.
\end{remark}

In the rest of this paper, we let $A = (A,\cdot,\circ)$ denote a skew brace. As usual, for each $a\in A$, we consider the map
\[ \lambda_a : A \longrightarrow A ;\,\ \lambda_a(b) = a^{-1}\cdot (a\circ b),\]
which is easily verified to be an automorphism on $(A,\cdot)$. Moreover, let
\[ \lambda : (A,\circ) \longrightarrow \Aut(A,\cdot);\,\ a\mapsto \lambda_a\]
denote the so-called \textit{lambda map} of $A$, which is known to be a group homomorphism \cite[Proposition 1.9]{GV}. For each $a,b\in A$, let us also define
\[ a*b = a^{-1}\cdot (a\circ b) \cdot b^{-1} = \lambda_a(b)\cdot b^{-1}.\]
This so-called \textit{asterisk} or \textit{star product}, which was introduced and first studied in \cite[Section 2]{KSV}, is a tool that measures the difference between the two group operations $\cdot$ and $\circ$ attached to the skew brace. 

\section{Preliminaries on skew braces}\label{sec:prelim}

In this section, we recall some basic definitions and simple facts that we shall need.

\begin{definition} A subset $B$ of $A$ is a \textit{sub-skew brace} if it is a subgroup of both $(A,\cdot)$ and $(A,\circ)$. In this case, clearly $(B,\cdot,\circ)$ is also a skew brace.
\end{definition}

\begin{definition}\label{def:left ideal} A subset $I$ of $A$ is a \textit{left ideal} if it is a subgroup of $(A,\cdot)$ and $\lambda_a(I) \subseteq I$, for all $a\in A$.
\end{definition}

Let us make a few remarks. First, given that $I$ is a subgroup of $(A,\cdot)$, for any $a\in A$ and $x\in I$, it is clear that
\[ \lambda_a(x) \in I \,\ \iff \,\ \lambda_a(x)x^{-1}\in I \,\ \iff \,\ a*x\in I.\]
So the terminology ``left ideal" in fact comes from ring theory. Second, a left ideal $I$ of $A$ is automatically a sub-skew brace because 
\[ x \circ y = x\cdot \lambda_x(y)\quad \mbox{and}\quad \overline{x} =\lambda_{\overline{x}}(x)^{-1}\]
for all $x,y\in A$. Finally, since
\[ a\cdot x = a\circ \lambda_{\overline{a}}(x)\quad \mbox{and}\quad a\circ x = a\cdot \lambda_a(x)\]
for all $a,x\in A$, a left ideal $I$ of $A$ always satisfies
\begin{equation}\label{aI}
    a\cdot I = a\circ I
\end{equation}
for all $a\in A$, namely, the left cosets of $I$ with respect to $\cdot$ and $\circ$ coincide.

\begin{definition}\label{def:ideal} A subset $I$ of $A$ is an \textit{ideal} if it is a left ideal of $A$ and is a normal subgroup of both $(A,\cdot)$ and $(A,\circ)$. In this case, we can naturally define a quotient skew brace on the set of cosets
\[ A/I = \{a\cdot I : a\in A \} = \{a\circ I : a\in A\}\]
by (\ref{aI}) and the normality condition.
\end{definition}

Note that ideals are exactly the sub-structures that one needs to form quotient skew braces. The terminology ``ideal" comes from ring theory, and ideals in skew braces are analogs of normal subgroups in groups.

In fact, historically, braces were introduced as a generalization of radical rings \cite{Rump0}, as follows. Given a (non-unital) ring $(A,+,\star)$, define
\[ a\circ b = a + b + a\star b,\]
for all $a,b\in A$. This operation $\circ$ is always associative and admits the zero element as the identity. Recall that $(A,+,\star)$ is said to be \textit{radical} if $(A,\circ)$ is a group, and in this case, we easily check that $(A,+,\circ)$ is a brace. We remark that a non-trivial unital ring cannot be radical because the additive inverse of the identity does not admit an inverse with respect to $\circ$. As is known by \cite{Rump0}, radical rings correspond precisely to \textit{two-sided braces}, i.e. braces for which the right brace relation
\[ (b\cdot c)\circ a = (b\circ a)\cdot a^{-1} \cdot (c\circ a),\]
holds for all $a,b,c\in A$ in addition to \eqref{brace relation}.

\begin{example} For a skew brace $A = (A,+,\circ)$ that is defined, via the above construction, from a radical ring $(A,+,\star)$, we have
\[ \lambda_a : A\longrightarrow A;\,\ \lambda_a(b) = b + a\star b,\]
for every $a\in A$. Hence, the left ideals of $A$ coincide with the left ideals of the underlying radical ring in this case. Moreover, we have
\[ a*b = -a + (a+b+a\star b) -b = a \star b,\]
for all $a,b\in A$, so the star product coincides with the multiplication of the underlying radical ring.
\end{example}

\begin{example}\label{ex:trivial} For a trivial skew brace $A = (A,\cdot,\cdot)$, the map
\[ \lambda_a: A \longrightarrow A;\,\ \lambda_a(b) = b \]
is the identity map, for every $a\in A$. Thus, the left ideals of $A$ coincide with the subgroups of $(A,\cdot)$ in this case. Moreover, we have
\[ a * b = a^{-1}\cdot (a\cdot b )\cdot b^{-1} = 1,\]
for all $a,b\in A$, so the star product only returns the identity.
\end{example} 

\begin{example}\label{ex:almost trivial} For an almost trivial skew brace $A = (A,\cdot,\cdot^{\mbox{\tiny op}})$, the map
\[ \lambda_a : A \longrightarrow A;\,\ \lambda_a(b) = a^{-1}\cdot b \cdot a\]
is conjugation by $a^{-1}$, for every $a\in A$. Thus, the left ideals of $A$ coincide with the normal subgroups of $(A,\cdot)$ in this case. Moreover, we have
\[ a * b  =a^{-1}\cdot b\cdot a \cdot b^{-1} = [a^{-1},b], \]
for all $a,b\in A$, so the star product is basically the commutator of $(A,\cdot)$. 
\end{example}

Below, we collect some well-known identities. The proofs are all straightforward, so we shall state them without references.

\begin{lemma}\label{lem:identities} For any $a,x,y\in A$, we have the identities
\begin{enumerate}[$(1)$]
\item $ a* (x\cdot y) = (a* x)\cdot x\cdot (a* y)\cdot x^{-1} ;$
\item $(x\circ y) * a = (x*(y*a))\cdot (y*a)\cdot (x*a);$
\item $\lambda_a(x*y) = (a\circ x\circ \overline{a}) * \lambda_a(y);$
\item $a\circ x\circ \overline{a} = a\cdot \lambda_a(x\cdot(x*\overline{a}))\cdot a^{-1}$.
\end{enumerate}
\end{lemma}


\section{Analogs of the lower central series}\label{sec:lower}

For a group $G$, the \textit{lower central series} is defined by
\[ \gamma_1(G) = G,\quad \gamma_{n+1}(G) = [G,\gamma_n(G)]\]
in terms of the commutator operator $[ \,\ , \,\ ]$. Using the star product, this may be naturally extended to skew braces. For any subsets $X,Y$ of $A$, we define $X*Y$ to be the subgroup of $(A,\cdot)$ generated by the elements $x*y$, for $x\in X$ and $y\in Y$. However, the star product $*$ is not a commutative operation on sub-skew braces, so after forming $A*A$, it matters whether we $*$-multiply $A$ from the left or the right. In other words, the star product $*$ is not an associative operation on skew braces, and there are two different natural analogs of the lower central series for a skew brace, as follows.

The \textit{left series} of $A$ is defined by
\[ A^{1} = A, \quad A^{n+1} = A*A^{n},\]
and the \textit{right series} of $A$ is defined by
\[ A^{(1)}=A, \quad A^{(n+1)} = A^{(n)} * A,\]
for all $n\in \mathbb{N}$. They were introduced by Rump in \cite{Rump0} for braces and were later extended to all skew braces in \cite{nilpotent}. In the case that $A = (A,\cdot,\cdot^{\mbox{\tiny op}})$ is an almost trivial skew brace, the left and right series are simply the lower central series of $(A,\cdot)$. However, in general, the left and right series are not even comparable, as the next example shows.

\begin{example}\label{ex:pq} Let $p,q$ be any primes with $p\equiv 1\hspace{-1mm}\pmod{q}$. The skew 
braces of order $pq$ were classified by \cite{pq}. Here we consider two of them. Let
\[ C_p = \langle a \rangle \mbox{ and } C_q = \langle b\rangle, \]
respectively, denote the cyclic groups of order $p$ and $q$. Let $k$ be any integer of multiplicative order $q$ modulo $p$. 
\begin{enumerate}[(i)]
\item First consider the brace $A = (C_p\times C_q,\cdot,\circ)$, where
\begin{align*}
 (a^i,b^j) \cdot (a^s, b^t) & = (a^{i+s},b^{j+t})\\
 (a^{i},b^{j}) \circ (a^{s}, b^{t}) & = (a^{i+k^{j}s},b^{j+t})
\end{align*}
for all $i,j,s,t\in\mathbb{Z}$. Since
\begin{equation}\label{pq*}
(a^{i},b^{j}) * (a^{s}, b^{t}) = (a^{(k^{j}-1)s},1),
\end{equation}
we see that $A*A= C_p\times \{1\}$. Moreover, observe that
\begin{align*}
    (a^{i},b^{j}) * (a^{s}, 1) &= (a^{(k^{j}-1)s},1),\\
    (a^{i},1) * (a^{s}, b^{t}) &= (1,1).
\end{align*}
It follows that the left series of $A$ is given by
\[ A^1 = C_p\times C_q,\quad A^n =C_p\times\{1\}\mbox{ for }n\geq 2,\]
while the right series of $A$ is given by
\[ A^{(1)}=C_p\times C_q,\quad A^{(2)}=C_p\times \{1\},\quad A^{(n)}=1\mbox{ for }n\geq 3.\]
Here, the right series reaches the identity, but not the left series.
\item Next consider the skew brace $A = (C_p\times C_q,\cdot,\circ)$, where
\begin{align*}
 (a^{i},b^{j}) \cdot (a^{s}, b^{t}) & = (a^{i+ k^{j}s},b^{j+t})\\
 (a^{i},b^{j}) \circ (a^{s}, b^{t}) & = (a^{k^{t}i+k^{j}s},b^{j+t})
\end{align*}
for all $i,j,s,t\in\mathbb{Z}$. Since
\begin{align*}
    & \quad \hspace{1mm} (a^{i},b^{j}) * (a^{s}, b^{t})\\
    & = (a^{-k^{-j}i},b^{-j})\cdot (a^{k^{t}i+k^{j}s},b^{j+t})\cdot (a^{-k^{-t}s},b^{-t})\\
    & = (a^{k^{-j}(k^{t}-1)i+s},b^{t})\cdot (a^{-k^{-t}s},b^{-t})\\
    & = (a^{k^{-j}(k^{t}-1)i},1),
\end{align*}
we see that $A*A= C_p\times \{1\}$. Moreover, observe that
\begin{align*}
    (a^{i},b^{j}) * (a^{s}, 1) &= (1,1),\\
    (a^{i},1) * (a^{s}, b^{t}) &= (a^{(k^{t}-1)i},1).
\end{align*}
It follows that the left series of $A$ is given by
\[ A^1 = C_p\times C_q,\quad A^2 = C_p\times \{1\},\quad A^n=1\mbox{ for }n\geq 3, \]
while the right series of $A$ is given by
\[ A^{(1)}=C_p\times C_q,\quad A^{(n)}= C_p\times \{1\} \mbox{ for }n\geq 2.\]
Here, the left series reaches the identity, but not the right series.
\end{enumerate}
\end{example}

The terms in the lower central series of a group are all normal subgroups. However, as is known by \cite[Propositions 2.1 and 2.2]{nilpotent}, the terms in the left series of a skew brace are only left ideals in general, while those in the right series are always ideals. 

\begin{proposition}\label{prop:left} For all $n\geq 1$, we have that $A^n$ is a left ideal of $A$.
\end{proposition}
\begin{proof} The case $n=1$ is trivial. Now, suppose that $A^n$ is a left ideal of $A$. For any $a,b\in A$ and $x\in A^n$, we then have $\lambda_a(x) \in A^n$, and so by Lemma \ref{lem:identities}(3) we have
\[ \lambda_a(b*x) = (a\circ b \circ \overline{a}) * \lambda_a(x) \in A^{n+1}.\]
Since the elements $b*x$ generate $A^{n+1}$ with respect to $\cdot$ and $\lambda_a\in \mathrm{Aut}(A,\cdot)$, we deduce that $\lambda_a(A^{n+1}) \subseteq A^{n+1}$ , for all $a\in A$. Hence, indeed $A^{n+1}$ is a left ideal of $A$.
\end{proof}

\begin{proposition}\label{prop:right} For all $n\geq1$, we have that $A^{(n)}$ is an ideal of $A$.
\end{proposition}
\begin{proof} The case $n=1$ is trivial. Now, suppose that $A^{(n)}$ is an ideal of $A$. 
\begin{enumerate}[(1)]
\item $A^{(n+1)}$ is a left ideal of $A$: For any $a,b\in A$ and $x\in A^{(n)}$, since $A^{(n)}$ is normal in $(A,\circ)$, we have $a\circ x \circ \overline{a} \in A^{(n)}$, and so by Lemma \ref{lem:identities}(3) we have
\[ \lambda_a(x*b) = (a\circ x \circ \overline{a}) * \lambda_a(b) \in A^{(n+1)}.\]
Since the elements $x*b$ generate $A^{(n+1)}$ with respect to $\cdot$ and $\lambda_a\in \mathrm{Aut}(A,\cdot)$, this yields $\lambda_a(A^{(n+1)}) \subseteq A^{(n+1)}$, for all $a\in A$. Hence, indeed $A^{n+1}$ is a left ideal of $A$.
\item $A^{(n+1)}$ is normal in $(A,\cdot)$: For any $a,b\in A$ and $x\in A^{(n)}$, by Lemma \ref{lem:identities}(1) we have
\begin{align*}
a \cdot (x*b) \cdot a^{-1}
& = (x*a)^{-1}\cdot (x*a)\cdot a \cdot (x*b) \cdot a^{-1}\\
& = (x*a)^{-1} \cdot (x*(a\cdot b))\in A^{(n+1)}. 
\end{align*}
Since the elements $x*b$ generate $A^{(n+1)}$ with respect to $\cdot$, it then follows that $A^{(n+1)}$ is normal in $(A,\cdot)$.
\item $A^{(n+1)}$ is normal in $(A,\circ)$: For any $a\in A$ and $y\in A^{(n+1)}$, clearly
\[ y(y*\overline{a}) \in A^{(n+1)}A^{(n+2)} \subseteq A^{(n+1)}.\]
Since $A^{(n+1)}$ is a left ideal of $A$ and is normal in $(A,\cdot)$, by Lemma \ref{lem:identities}(4) we have
\[a\circ y \circ \overline{a}
= a \cdot \lambda_a(y\dot (y* \overline{a})) \cdot a^{-1} \in A^{(n+1)},\]
and this proves that $A^{(n+1)}$ is normal in $(A,\circ)$.
\end{enumerate}
We have thus shown that $A^{(n+1)}$ is an ideal of $A$.    
\end{proof}

We remark that the left and right series of $A$ are not comparable with the lower central series of $(A,\cdot)$ and $(A,\circ)$ in general. Simply take $A = (A,\cdot,\cdot)$ to be a trivial skew brace such that $(A,\cdot)$ is perfect, in which case 
\[ A^n = A^{(n)} = 1,\quad \gamma_n(A,\cdot) = A\]
for all $n\geq 2$. In the other extreme, recall that there are simple braces $A = (A,\cdot,\circ)$ of non-prime order by \cite{simple}, in which case $A*A=A$ and
\[ A^n = A^{(n)}=A,\quad \gamma_n(A,\cdot) = 1\]
for all $n\geq 2$. There should be no surprise about these examples as the star product is not defined in terms of the commutator of the underlying groups.

Here, for simplicity, we only considered the left series and the right series. But instead of $*$-multiplying $A$ only from the left or only from the right, we can do a mixture of both, as follows.

The \textit{Smoktunowicz series} of $A$ is defined by
\[ A^{[1]} = A,\quad A^{[n+1]} = \left\langle
\bigcup_{i=1}^{n} A^{[i]} * A^{[n+1-i]}
\right\rangle,\]
for all $n\in\mathbb{N}$, where the angle brackets $\langle \,\ \rangle$ denote subgroup generation in $(A,\cdot)$. This was introduced by Smoktunowicz in \cite{Smok} for braces and was extended to all skew braces in \cite{nilpotent}. Its terms are all left ideals of $A$ by \cite[Proposition 2.28]{nilpotent}. Let us mention that the Smoktunowicz series is closely related to the left and right series by the following theorem (see \cite[Theorem 1.3]{Smok} for the case of braces and \cite[Theorem 2.30]{nilpotent} for the general case).

\begin{theorem}The following statements are equivalent.
\begin{enumerate}[$(a)$]
\item $A^{[n]}=1$ for some $n\in\mathbb{N}$.
\item $A^{m_1}=1$ and $A^{(m_2)}=1$ for some $m_1,m_2\in\mathbb{N}$.
\end{enumerate}
\end{theorem}

Finally, there is another analog of the lower central series for skew braces that was defined in \cite{central}. It seems to be the better analog for various reasons, and we shall say more about it in Section \ref{sec:central}.

\section{Analogs of the upper central series}\label{sec:upper}

For a group $G$, the \textit{upper central series} is defined by
\[ \zeta_0(G) = 1,\quad Z(G/\zeta_{n}(G)) = \zeta_{n+1}(G)/\zeta_n(G)\]
in terms of the center operator $Z(\,\ )$. Note that the center
\[ Z(G) = \{x\in G \mid [x,g]=1 \mbox{ for all }g\in G\} \]
of $G$ may be defined in terms of the commutator operator $[\,\ ,\,\ ]$. Using the star product, this may be naturally extended to skew braces. But not only that $x*a=1$ and $a*x=1$ are not equivalent for $a,x\in A$ in general, there are also other technicalities. Here, we introduce two natural analogs of the center that are frequently used in the literature, as follows.

The \textit{socle} of $A$ is defined by
\begin{align}\notag
\mathrm{Soc}(A) &= \{x\in A\mid x*a = 1\mbox{ for all }a\in A\}\cap Z(A,\cdot)\\\label{soc def}
& = \ker(\lambda)\cap Z(A,\cdot),
\end{align} 
and the \textit{annihilator} of $A$ is defined by
\begin{align}\notag
    \mathrm{Ann}(A) &= \{x\in A\mid x*a = a*x = 1\mbox{ for all }a\in A\}\cap Z(A,\cdot)\\\label{Ann def}
    & = \mathrm{Soc}(A)\cap Z(A,\circ)\\\notag
    & = \ker(\lambda)\cap Z(A,\cdot)\cap Z(A,\circ).
\end{align} These two analogs of the center, which first appeared in  \cite[Definition 2.4]{GV} and \cite[Definition 7]{ann}, respectively, are better than the others in some sense because both of them are ideals of $A$, so in particular we can form quotients. This is well-known, but we shall give a proof for the sake of completeness.

\begin{proposition}\label{prop:Soc} The socle $\mathrm{Soc}(A)$ is an ideal of $A$.
\end{proposition}
\begin{proof} Since $\lambda$ is a homomorphism on $(A,\circ)$, its kernel $\ker(\lambda)$ is a subgroup of $(A,\circ)$. But $\cdot$ and $\circ$ coincide on $\ker(\lambda)$. It then follows that $\ker(\lambda)$, and in particular $\mathrm{Soc}(A)$ in view of (\ref{soc def}), is a subgroup of $(A,\cdot)$.   
\begin{enumerate}[(1)]
\item $\mathrm{Soc}(A)$ is a left ideal of $A$: For any $a\in A$ and $x\in \mathrm{Soc}(A)$, that
\[ \lambda_a(x)\in Z(A,\cdot)\]
is clear because $\lambda_a\in\mathrm{Aut}(A,\cdot)$. Note that $x*\overline{a}=1$ because $x\in \ker(\lambda)$. By Lemma~\ref{lem:identities}(4), we then see that
\begin{align}\notag
a \circ x \circ \overline{a} & = a\lambda_a(x(x*\overline{a}))a^{-1}\\\notag
& = a \lambda_a(x)a^{-1}\\\label{socle equation}
& = \lambda_a(x).
\end{align}
But $\ker(\lambda)$ is a normal subgroup of $(A,\circ)$, so we deduce that
\[
\lambda_a(x)\in\ker(\lambda)\]
also holds. Thus, we have shown that $\lambda_a(x) \in \mathrm{Soc}(A)$, so $\mathrm{Soc}(A)$ is a left ideal of $A$.
\item $\mathrm{Soc}(A)$ is normal in $(A,\cdot)$: This is because $\mathrm{Soc}(A)\subseteq Z(A,\cdot)$.
\item $\mathrm{Soc}(A)$ is normal in $(A,\circ)$: This follows from (\ref{socle equation}).
\end{enumerate}
Hence, indeed $\mathrm{Soc}(A)$ is an ideal of $A$.
\end{proof}

\begin{proposition}\label{prop:Ann} The annihilator $\mathrm{Ann}(A)$ is an ideal of $A$.
\end{proposition}
\begin{proof} We know that $\mathrm{Soc}(A)$ is a subgroup of $(A,\circ)$ by Proposition \ref{prop:Soc}, so clearly $\mathrm{Ann}(A)$ is also a subgroup of $(A,\circ)$ by (\ref{Ann def}). Since $\cdot$ and $\circ$ coincide on $\ker(\lambda)$, it follows that $\mathrm{Ann}(A)$ is a subgroup of $(A,\cdot)$.
\begin{enumerate}[(1)]
\item $\mathrm{Ann}(A)$ is a left ideal of $A$: For any $a\in A$ and $x\in \mathrm{Ann}(A)$, we have
\[ \lambda_a(x) = (a*x)\cdot x = x,\]
so trivially $\mathrm{Ann}(A)$ is a left ideal of $A$.
\item $\mathrm{Ann}(A)$ is normal in $(A,\cdot)$: This is because $\mathrm{Ann}(A)\subseteq Z(A,\cdot)$.
\item $\mathrm{Ann}(A)$ is normal in $(A,\circ)$: This is because $\mathrm{Ann}(A)\subseteq Z(A,\circ)$.
\end{enumerate}
Hence, indeed $\mathrm{Ann}(A)$ is an ideal of $A$.
\end{proof}

The \textit{socle series} of $A$ is defined by
\begin{equation}\label{socle series} \mathrm{Soc}_0(A) = 1, \quad \mathrm{Soc}(A/\mathrm{Soc}_n(A)) = \mathrm{Soc}_{n+1}(A)/\mathrm{Soc}_n(A),
\end{equation}
and the \textit{annihilator series} of $A$ is defined by
\begin{equation}\label{ann series}
 \mathrm{Ann}_0(A) = 1, \quad \mathrm{Ann}(A/\mathrm{Ann}_n(A)) = \mathrm{Ann}_{n+1}(A)/\mathrm{Ann}_n(A).
\end{equation}
For each $n\geq0 $, it follows from Propositions \ref{prop:Soc}, \ref{prop:Ann}, and the lattice theorem in groups that $\mathrm{Soc}_{n+1}(A)$ and $\mathrm{Ann}_{n+1}(A)$ are uniquely determined ideals of $A$. Thus, we can indeed recursively form the skew brace quotients
\[ A/ \mathrm{Soc}_n(A),\quad A/\mathrm{Ann}_n(A).\]
The socle series was introduced by Rump \cite{Rump0} for braces and was extended to all skew braces in \cite{nilpotent}, while the annihilator series first appeared in \cite{central}. In the case that $A=(A,\cdot,\cdot)$ is trivial or $A = (A,\cdot,\cdot^{\mbox{\tiny op}})$ is almost trivial, the socle and annihilator series are the upper central series of $(A,\cdot)$. 

\begin{remark}
    The authors of \cite{SV} defined the socle series of $A$ by
\begin{equation}\label{socle series'} A_1 = A,\quad A_{n+1} = A_n/\mathrm{Soc}(A_n)\end{equation}
    instead, and the \texttt{SocleSeries} command in \texttt{GAP} \cite{GAP} returns this series \eqref{socle series'} rather than the series \eqref{socle series}. But observe that by induction, we have a skew brace isomorphism, i.e. a map that is an automorphism with respect to both of the operations $\cdot$ and $\circ$, between
    \[ A_{n} \simeq A/\mathrm{Soc}_{n-1}(A),\]
    for all $n\geq 1$. Indeed, for $n=1$, this is trivial. Assuming that we have such an isomorphism for $n$, we see that
    \begin{align*}
        A_{n+1} & = A_n/\mathrm{Soc}(A_n)\\
        & \simeq (A/\mathrm{Soc}_{n-1}(A))/\mathrm{Soc}(A/\mathrm{Soc}_{n-1}(A))\\
        & = (A/\mathrm{Soc}_{n-1}(A))/(\mathrm{Soc}_n(A)/\mathrm{Soc}_{n-1}(A))\\
        & \simeq A/\mathrm{Soc}_{n}(A).
    \end{align*} 
Hence, the terms in the series (\ref{socle series'}) are nothing but quotients of those in the socle series of our definition, and there is no fundamental difference between the two series. 
\end{remark}

Unlike the left and right series, the socle and annihilator series are comparable. Indeed, a simple induction on $n$ shows that
\begin{equation}\label{AnnSoc} \mathrm{Ann}_n(A) \subseteq \mathrm{Soc}_n(A)\end{equation}
for all $n\geq 0$. But the other inclusion need not hold in general.

\begin{example}\label{ex:pq'} Let $A = (C_p\times C_q,\cdot,\circ)$ be the  brace in Example \ref{ex:pq}(i). From (\ref{soc def}), we see that $\mathrm{Soc}(A)= \ker(\lambda)$ because $(A,\cdot)$ is abelian. Moreover, from (\ref{pq*}), it is easy to see that $\mathrm{Soc}(A) = C_p\times \{1\}$ and the quotient $A/\mathrm{Soc}(A)$ is a trivial brace. This means that
\[ \mathrm{Soc}(A/\mathrm{Soc}(A)) = Z(A/\mathrm{Soc}(A),\cdot) = A/\mathrm{Soc}(A).\]
It follows that the socle series of $A$ is given by
\[ \mathrm{Soc}_0(A)=1,\quad \mathrm{Soc}_1(A)=C_p\times\{1\},\quad \mathrm{Soc}_n(A)=A\mbox{ for }n\geq 2. \]
On the other hand, note that $\mathrm{Ann}(A)=1$ since $(A,\circ)$ has a trivial center. It follows that the annihilator series of $A$ is given by
\[ \mathrm{Ann}_n(A)=1\mbox{ for }n\geq 0.\]
Thus, the socle and annihilator series are different here.
\end{example}

As one can expect from their definitions, the socle and annihilator series of $A$, respectively, are comparable with the upper central series of $(A,\cdot)$ and $(A,\circ)$. More precisely, a simple induction on $n$ yields that
\begin{equation}\label{Soc zeta} \mathrm{Ann}_n(A)\subseteq \mathrm{Soc}_n(A) \subseteq \zeta_n(A,\cdot)\quad\mbox{and}\quad
\mathrm{Ann}_n(A)\subseteq \zeta_n(A,\circ),\end{equation}
for all $n\geq 0$. But in general, the socle series of $A$ and the upper central series of $(A,\circ)$ are not comparable. Indeed, for the brace $A = (C_p\times C_q,\cdot,\circ)$ in Example \ref{ex:pq}(i),  by Example~\ref{ex:pq'} and the fact that $(A,\circ)$ is centerless, we have
\[ \mathrm{Soc}_n(A) = A \quad\mbox{and}\quad \zeta_n(A,\circ) = 1,\]
for all $n\geq 2$. But for the skew brace $A=(C_p\times C_q,\cdot,\circ)$ in Example \ref{ex:pq}(ii), we have
\[ \mathrm{Soc}_n(A) = 1\quad \mbox{and} \quad \zeta_n(A,\circ)=A,\]
for all $n\geq 1$ because $(A,\cdot)$ is centerless and $(A,\circ)$ is abelian.

\section{Comparison of the different types of nilpotency}\label{sec:compare}

For a group $G$, it is called \textit{nilpotent} if its lower central series reaches $1$, or equivalently, if its upper central series reaches $G$. Using the four series from Sections \ref{sec:lower} and \ref{sec:upper}, we can naturally extend the notion of nilpotency to skew braces, as follows.

\begin{definition}\label{def:nilpotent} The skew brace $A$ is said to be
\begin{enumerate}[(1)]
\item \textit{left nilpotent} if $A^n=1$ for some $n$;
\item \textit{right nilpotent} if $A^{(n)}=1$ for some $n$;
\item \textit{socle nilpotent} if $\mathrm{Soc}_n(A)= A$ for some $n$;
\item \textit{annihilator nilpotent} if $\mathrm{Ann}_n(A) =A$ for some $n$.
\end{enumerate}
We note that the term \textit{finite multipermutation level} is also used for (3), such as in \cite{nilpotent} due to its relation with the Yang-Baxter equation, while \textit{centrally nilpotent} is sometimes used for (4), such as in \cite{central, nil}. 
\end{definition}

For a group $G$, its lower central series reaches the identity if and only if its upper central series reaches $G$. However, for a skew brace, the situation is totally different. Simply take $A = (A,\cdot,\cdot)$ to be a trivial skew brace with non-trivial centerless underlying group $(A,\cdot)$, in which case $A$ is clearly left and right nilpotent, but is neither socle nilpotent nor annihilator nilpotent. We remark that left and right nilpotency are non-equivalent by Example \ref{ex:pq}, while socle and annihilator nilpotency are also non-equivalent by Example \ref{ex:pq'}.

Nevertheless, there are some relations among the four different types of nilpotency. For example, from (\ref{AnnSoc}), it is clear that annihilator nilpotent implies socle nilpotent. Moreover, it was recently proven in \cite{BKP} that a left nilpotent skew brace $A$ with nilpotent additive group $(A,\cdot)$ and $A^3=1$ is necessarily right nilpotent. But as we can see from Example~\ref{ex:pq}, in general, there is no implication between left nilpotent and right nilpotent. Here, we shall survey a couple of known results about the various types of nilpotency.

Theorem \ref{thm:nilpotency1} below is from \cite[Theorem 2.20]{nilpotent}. We give a slightly different and completely self-contained proof here.

\begin{theorem}\label{thm:nilpotency1} The following are equivalent:
\begin{enumerate}[$(a)$]
\item $A$ is right nilpotent and $(A,\cdot)$ is nilpotent.
\item $A$ is socle nilpotent.
\end{enumerate}
\end{theorem}
\begin{proof}
Assume that (a) holds. We have $A^{(n)}=1$ for some $n$ since $A$ is right nilpotent. For $n=1$, the claim is clear. For $n\geq 2$, we consider the quotient 
\[Q := A/A^{(n-1)},\]
which is possible by Proposition \ref{prop:right}. Since
\[ Q^{(n-1)} = A^{(n-1)}/A^{(n-1)} =1\]
and $(Q,\cdot)$ is certainly nilpotent, by induction, we can assume that $Q$ is socle nilpotent. This means that
\[ 1 = \mathrm{Soc}_0(Q) \subseteq \mathrm{Soc}_1(Q) \subseteq \cdots \subseteq \mathrm{Soc}_{m-1}(Q) \subseteq \mathrm{Soc}_m(Q) = Q , \]
for some $m$. This chain of ideals of $Q$ lifts to a chain of ideals
\[ A^{(n-1)} = I_0 \subseteq I_1 \subseteq \cdots \subseteq I_{m-1} \subseteq I_m = A\]
of $A$, where for each $0\leq j\leq m$, we put
\[\mathrm{Soc}_j(Q) = I_j/A^{(n-1)}.\]
It is not hard to see that for each $0\leq j\leq m-1$, we have
\begin{equation}\label{Soc I}
\mathrm{Soc}(A/I_j) = I_{j+1}/I_j.
\end{equation}
In what follows, let $[\,\ ,\,\ ]$ denote the commutator in the group $(A,\cdot)$.
\begin{enumerate}[(1)]
\item We first use induction to show that 
\[ A^{(n-1)} \cap \zeta_\ell (A,\cdot) \subseteq \mathrm{Soc}_\ell(A) ,\]
for all $\ell\geq 0$. The case $\ell=0$ is trivial. Suppose that the inclusion holds for $\ell$, and let $x\in A^{(n-1)}\cap \zeta_{\ell+1}(A,\cdot)$. For any $a\in A$, we have 
\[ x*a=1 \in \mathrm{Soc}_\ell(A) \]
because $A^{(n)}=1$. Recall that $A^{(n-1)}$ is normal in $(A,\cdot)$ by Proposition \ref{prop:right}. Together with the definition of $\zeta_{\ell+1}(A,\cdot)$, we see that
\[ [x,a]\in A^{(n-1)}\cap \zeta_\ell(A,\cdot)\subseteq \mathrm{Soc}_\ell(A).\]
It follows that $x\in \mathrm{Soc}_{\ell+1}(A)$, as desired.
\end{enumerate}
Now, since $(A,\cdot)$ is nilpotent, we have $\zeta_c(A,\cdot)=A$ for some $c$.
\begin{enumerate}[(1)]\setcounter{enumi}{+1}
\item We next use induction to show that
\[I_j \subseteq \mathrm{Soc}_{c+j}(A),\]
for all $0\leq j\leq m$. For $j=0$, this is (1) in the case $\ell = c$. Suppose that the inclusion holds for $j$, and let $x\in I_{j+1}$. For any $a\in A$, cleary
\[ x*a,\, [x,a] \in I_{j} \subseteq 
 \mathrm{Soc}_{c+j}(A)\]
by (\ref{Soc I}), so it follows that $x\in \mathrm{Soc}_{c+j+1}(A)$, as desired.
\end{enumerate}
By taking $j=m$ in (2), we get that $\mathrm{Soc}_{c+m}(A)=A$, and this shows that $A$ is socle nilpotent, yielding (b).

Now, assume that (b) holds. We have $\mathrm{Soc}_n(A) = A$ for some $n$ from the hypothesis. Note that $(A,\cdot)$ is nilpotent by (\ref{Soc zeta}). We also have
\[ A^{(j+1)} \subseteq \mathrm{Soc}_{n-j}(A) ,\]
for all $0\leq j\leq n$. Indeed, the case $j=0$ is clear. Suppose that the inclusion holds for $j$. This implies that for all $a\in A$ and $x\in A^{(j+1)}$, we have
\[ x\in \mathrm{Soc}_{n-j}(A),\mbox{ and so } x*a \in \mathrm{Soc}_{n-(j+1)}(A).\]
Since these elements $x*a$ generate $A^{(j+2)}$ with respect to $\cdot$, it follows that the inclusion also holds for $j+1$. Taking $j=n$, we see that $A^{(n+1)}=1$, so indeed $A$ is right nilpotent, yielding (a).
\end{proof}

Theorem \ref{thm:nilpotency2} below is a combination of \cite[Proposition 2.12 \& Corollary 2.15]{nilpotent'} and \cite[Theorem 2.7]{central}. A similar statement, but under the assumption that $A$ is finite, appeared in \cite[Corollary 2.11]{central}. Finiteness is imposed there because part of its proof uses the result \cite[Theorem 4.8]{nilpotent}, which is stronger than what is really needed. Using the argument in \cite[Corollary 2.15]{nilpotent'} instead, one sees that finiteness can be dropped. Again, we shall give a self-contained proof here.

\begin{theorem}\label{thm:nilpotency2} The following are equivalent.
\begin{enumerate}[$(a)$]
\item $A$ is both left and right nilpotent, and $(A,\cdot)$ is nilpotent.
\item $A$ is right nilpotent, and both $(A,\cdot)$ and $(A,\circ)$ are nilpotent.
\item $A$ is annihilator nilpotent.
\end{enumerate}
\end{theorem}
\begin{proof}

First, suppose that $A$ is right nilpotent and $(A,\cdot)$ is nilpotent, which are the common hypotheses in (a) and (b). We have $\mathrm{Soc}_n(A) = A$ for some $n$ by Theorem \ref{thm:nilpotency1}. For $n=0$, clearly $(A,\circ)$ is nilpotent and $A$ is annihilator nilpotent, so there is nothing to prove. For $n\geq 1$, we consider the quotient
\[Q:=A/\mathrm{Soc}(A),\]
which is possible by Proposition \ref{prop:Soc}. Clearly $Q$ is right nilpotent, and $(Q,\cdot)$ is nilpotent. It is also easy to see that
\[ \mathrm{Soc}_\ell(Q) = \mathrm{Soc}_{\ell+1}(A)/\mathrm{Soc}(A),\]
for all $\ell\geq 0$, and so in particular $\mathrm{Soc}_{n-1}(Q)=Q$. In what follows, let $[\,\ ,\,\ ]$ and $[\,\ ,\,\ ]_\circ$, respectively, denote the commutator in $(A,\cdot)$ and $(A,\circ)$. Using this set-up, we shall now prove (a)$\Rightarrow$(b) and (b)$\Rightarrow$(c).

\underline{(a)$\Rightarrow$(b)}: Here we assume that $A$ is left nilpotent. Then $Q$ is also left nilpotent, so by induction, we can assume that $(Q,\circ)$ is nilpotent. This means that
\[ \gamma_m(A,\circ) \subseteq \mathrm{Soc}(A),\]
for some $m$. We use induction to show that
\[ \gamma_{m +j}(A,\circ) \subseteq \mathrm{Soc}(A) \cap A^{j+1},\]
for all $j\geq 0$. The case $j=0$ is trivial. Suppose that the inclusion holds for $j$. For any $a\in A$ and $x\in \gamma_{m+j}(A,\circ)$, we then have $x\in \mathrm{Soc}(A)$ and so
\[ [a,x]_\circ\in \mathrm{Soc}(A)\]
because $\mathrm{Soc}(A)$ is normal in $(A,\circ)$ by Proposition \ref{prop:Soc}. On the other hand, recall from (\ref{socle equation}) that $a\circ x \circ\overline{a} = \lambda_a(x)$ and this lies in $\ker(\lambda)$. This yields
 \begin{align*}
 [a,x]_\circ & = \lambda_a(x)\circ \overline{x}=\lambda_a(x)x^{-1}=a*x.
\end{align*}
But we know that $x\in A^{j+1}$, so this implies that
\[ [a,x]_\circ \in A^{j+2}.\]
Since these elements $[a,x]_\circ$ generate $\gamma_{m+j+1}(A,\circ)$ with respect to $\circ$, we see that the inclusion also holds for $j+1$. But $A$ is left nilpotent, so we deduce that $\gamma_{m+j}(A,\circ)=1$ for $j$ large enough, whence $(A,\circ)$ is nilpotent and this proves (b).

\underline{(b)$\Rightarrow$(c)}: Here we assume that $(A,\circ)$ is nilpotent. Then $(Q,\circ)$ is also nilpotent, so by induction, we can assume that $Q$ is annihilator nilpotent. This means that
\[ 1 = \mathrm{Ann}_0(Q) \subseteq \mathrm{Ann}_1(Q) \subseteq \cdots \subseteq \mathrm{Ann}_{m-1}(Q) \subseteq \mathrm{Ann}_m(Q) = Q ,\]
for some $m$. This chain of ideals of $Q$ lifts to a chain of ideals
\[ \mathrm{Soc}(A) = I_0 \subseteq I_1 \subseteq \cdots \subseteq I_{m-1} \subseteq I_m = A\]
of $A$, where for each $0\leq j\leq m$, we put
\[ \mathrm{Ann}_j(Q) = I_j/\mathrm{Soc}(A).\]
It is not hard to see that for each $0\leq j\leq m-1$, we have
\begin{equation}\label{Ann I}
\mathrm{Ann}(A/I_j) = I_{j+1}/I_j.
\end{equation}
We proceed as in the proof of Theorem \ref{thm:nilpotency1}.
\begin{enumerate}[(1)]
\item We first use induction to show that
\[ \mathrm{Soc}(A)\cap \zeta_\ell(A,\circ) \subseteq \mathrm{Ann}_\ell(A),\]
for all $\ell\geq 0$. The case $\ell=0$ is trivial. Suppose that the inclusion holds for $\ell$, and let $x\in \mathrm{Soc}(A)\cap\zeta_{\ell+1}(A,\circ)$. For any $a\in A$, we have
\[ x*a =[x,a] =  1 \in \mathrm{Ann}_{\ell}(A)\]
because $x\in \mathrm{Soc}(A)$. Recall that $\mathrm{Soc}(A)$ is normal in $(A,\circ)$ by Proposition \ref{prop:Soc}. Together with the definition of $\zeta_{\ell+1}(A,\circ)$, we see that 
\[ [x,a]_\circ \in \mathrm{Soc}(A) \cap \zeta_{\ell}(A,\circ)\subseteq\mathrm{Ann}_\ell(A).\]
It follows that $x\in \mathrm{Ann}_{\ell+1}(A)$, as desired.
\end{enumerate}
Now, since $(A,\circ)$ is nilpotent, we have $\zeta_c(A,\circ) = A$ for some $c$.
\begin{enumerate}[(1)]\setcounter{enumi}{+1}
\item We next use induction to show that
\[ I_j \subseteq \mathrm{Ann}_{c+j}(A) ,\]
for all $0\leq j\leq m$. For $j=0$, this is (1) in the case $\ell=c$. Suppose that the inclusion holds for $j$, and let $x\in I_{j+1}$. For any $a\in A$, we have
\[x*a, \, [x,a],\, [x,a]_\circ \in I_j \subseteq \mathrm{Ann}_{c+j}(A)\]
by (\ref{Ann I}), which means that $x\in \mathrm{Ann}_{c+j+1}(A)$, as desired.
\end{enumerate}
By taking $j=m$ in (2), we obtain $\mathrm{Ann}_{c+m}(A)=A$, and this shows that $A$ is annihilator nilpotent. Thus, indeed (c) holds.


\underline{(c)$\Rightarrow$(a)}: Finally, we assume that $A$ is annihilator nilpotent, that is $\mathrm{Ann}_n(A)=A$ for some $n$. Note that $(A,\cdot)$ is nilpotent by (\ref{Soc zeta}). We also have
\[ A^{j+1}\cup A^{(j+1)} \subseteq \mathrm{Ann}_{n-j}(A) ,\]
for all $0\leq j\leq n$. Indeed, for $j=0$, this is clear. Suppose that the inclusion holds for $j$. Then for all $a\in A$ and $x\in A^{j+1},\, y\in A^{(j+1)}$, we have
\[ x,y\in \mathrm{Ann}_{n-j}(A),\mbox{ and so } a * x, \, y * a \in \mathrm{Ann}_{n-(j+1)}(A).\]
Since these elements $a*x$ and $y*a$, respectively, generate $A^{j+2}$ and $A^{(j+2)}$ with respect to $\cdot$, the inclusion holds for $j+1$ as well. Taking $j=n$, we see that $A^{n+1} = A^{(n+1)}=1$, so then $A$ is both left and right nilpotent. 
\end{proof}

\section{Relationship among the terms in the series}\label{sec:relation}

For a group $G$, the terms in its lower and upper central series satisfy the inclusion in the following proposition. 

\begin{proposition}\label{prop:group} For a group $G$, we always have
\[ [\zeta_n(G),\gamma_{n-k}(G)] \subseteq \zeta_{k}(G),\]
for all $n\geq 1$ and $0\leq k \leq n-1$. 
\end{proposition}
\begin{proof}See \cite[5.1.11(iii)]{Robinson}; note that there is a typo there and the subscript should say $j-i$ rather than $j-1$.
\end{proof}

Here, we are interested in investigating the analog of Proposition \ref{prop:group} for skew braces. We considered two analogs each for the lower and upper central series. Since the star product is not commutative on sub-skew braces, there are eight different natural analogs, as follows:
\begin{enumerate}[(A)]
\itemsep=-2pt
\item $\mathrm{Soc}_n(A)* A^{n-k} \subseteq \mathrm{Soc}_k(A)$;
\item $\mathrm{Soc}_n(A)* A^{(n-k)} \subseteq \mathrm{Soc}_k(A)$;
\item $A^{n-k}*\mathrm{Soc}_n(A) \subseteq \mathrm{Soc}_k(A)$;
\item $A^{(n-k)}*\mathrm{Soc}_n(A) \subseteq \mathrm{Soc}_k(A)$;
\item $\mathrm{Ann}_n(A)* A^{n-k} \subseteq \mathrm{Ann}_k(A)$;
\item $\mathrm{Ann}_n(A)* A^{(n-k)} \subseteq \mathrm{Ann}_k(A)$;
\item $A^{n-k}*\mathrm{Ann}_n(A) \subseteq \mathrm{Ann}_k(A)$;
\item $A^{(n-k)}*\mathrm{Ann}_n(A) \subseteq \mathrm{Ann}_k(A)$.
\end{enumerate}
As we shall now explain, it turns out that only (E) is true for all $n\geq 1$ and $0\leq k\leq n-1$ in general. 

For $(n,k) = (2,0)$, the inclusions (A),(B) both state that
\[ \mathrm{Soc}_2(A) * (A*A) =1.\]
For $(n,k)=(1,0)$, the inclusions (C),(D) both state that
\[A * \mathrm{Soc}(A) =1.\]
The next example shows that these fail in general.

\begin{example}
Let $A = (C_p\times C_q,\cdot,\circ)$ be the brace in Examples \ref{ex:pq}(i) and \ref{ex:pq'}. Both of
\begin{align*}
    \mathrm{Soc}_2(A) *(A*A) & = A*(A*A) = A^3 = C_p\times \{1\}\\
    A*\mathrm{Soc}(A) &= A*(A*A) =  A^3 = C_p\times \{1\}
\end{align*} 
are non-singleton based on what we have already computed.
\end{example}

For $(n,k)=(2,0)$, the inclusions (G),(H) both state that
\[ (A*A)*\mathrm{Ann}_2(A)=1,\]
which does not hold in general by the examples given in \cite[Section 3]{TsangGrun}. 

For $(n,k)=(2,0)$, the inclusions (E),(F) both state that
\[ \mathrm{Ann}_2(A) * (A*A) =1,\]
which always holds by \cite[Proposition 2.2]{TsangGrun}. Below, we shall extend its proof to show (E). The idea is very similar to the proof of Proposition \ref{prop:group} that we gave in version 1 of this paper on the arXiv.

\begin{theorem}\label{main thm} We always have
\[\mathrm{Ann}_n(A)* A^{n-k} \subseteq \mathrm{Ann}_k(A),\]
for all $n\geq 1$ and $0\leq k\leq n-1$.
\end{theorem}
\begin{proof} For $n=1$, the claim simply states that
\[ \mathrm{Ann}(A) * A = 1,\]
which is trivial. For $n\geq 2$, suppose that the claim holds for $n-1$, and we shall use descending induction on $k$ to prove the claim.

For $k = n-1$, the claim simply states that
\[ \mathrm{Ann}_n(A)*A\subseteq \mathrm{Ann}_{n-1}(A), \]
which is clear. For $0\leq k \leq n-2$, let $a\in \mathrm{Ann}_n(A)$ and consider the map
\[ \varphi_{a}: (A^{n-(k+1)},\cdot) \longrightarrow (A/\mathrm{Ann}_k(A),\cdot);\,\ \varphi_{a}(y) = (a*y)\mathrm{Ann}_k(A).
\]
Let $x\in A$ and $y\in A^{n-(k+1)}$. We have
\begin{equation}\label{induction*}
a*y \in \mathrm{Ann}_n(A)*A^{n-(k+1)}   \subseteq \mathrm{Ann}_{k+1}(A)
\end{equation}
by induction on $k$. From Lemma \ref{lem:identities}(1), we then deduce that $\varphi_{a}$ is a homomorphism. Observe that we have
\begin{align}\notag
\varphi_a(x*y) &= \varphi_a(\lambda_x(y)\cdot y^{-1})\\\label{varphi break}
& = \varphi_a(\lambda_x(y)) \varphi_a(y)^{-1}\\\notag
& = (a*\lambda_x(y))(a*y)^{-1}\mathrm{Ann}_{k}(A)
\end{align}
because $\lambda_x(y)\in A^{n-(k+1)}$ by Proposition \ref{prop:left}. But note that
\[ \overline{a}\circ x\circ a \circ \overline{x} \in \mathrm{Ann}_{n-1}(A)\]
because $a\in \mathrm{Ann}_n(A)$, and by induction on $n$, we have 
\[ \mathrm{Ann}_{n-1}(A)* A^{(n-1)-k} \subseteq \mathrm{Ann}_k(A). \]
From these observations and  Lemma \ref{lem:identities}(2),(3), we see that
\begin{align*}
a* \lambda_x(y) \equiv (a \circ (\overline{a}\circ x\circ a \circ \overline{x})) * \lambda_x(y)  \equiv \lambda_x(a*y) \pmod{\mathrm{Ann}_k(A)},
\end{align*}
and together with (\ref{induction*}), this yields
\begin{align*}
    \varphi_a(x*y) & = \lambda_x(a*y) (a*y)^{-1}\mathrm{Ann}_k(A)\\
    &= (x*(a*y))\mathrm{Ann}_k(A) \\
    & = \mathrm{Ann}_k(A).
\end{align*} 
Since these elements $x*y$ generate $A^{n-k}$ with respect to $\cdot$, we see that
\[ \varphi_a(z) = \mathrm{Ann}_k(A), \mbox{ namely } a*z \in \mathrm{Ann}_k(A),\]
for all $z\in A^{n-k}$, and $a\in \mathrm{Ann}_n(A)$ was arbitrary. Hence, we get the desired inclusion, and this completes the proof of the theorem.
\end{proof}

The proof of Theorem \ref{main thm} cannot be modified to prove (F). In the induction step, for $a\in \mathrm{Ann}_n(A)$ we can similarly consider the map
\[ \psi_a : (A^{(n-(k+1))},\cdot) \longrightarrow (A/\mathrm{Ann}_k(A),\cdot);\,\ \psi_a(y) = (a*y)\mathrm{Ann}_k(A), \]
which can be assumed to be a homomorphism by induction on $k$ as in (\ref{induction*}). In order to prove (F), we would have to show that
\[ \psi_a(y*x) = \mathrm{Ann}_k(A),\]
for all $x\in A$ and $y\in A^{(n-(k+1))}$. But we cannot split
\[ \psi_a(y*x) = \psi_a(\lambda_y(x)x^{-1}) = \psi_a(\lambda_y(x)) \psi_a(x)^{-1}\]
as in (\ref{varphi break}) because $\lambda_y(x)$ and $x$ are not always elements of $A^{(n-(k+1))}$ here. There is no way to amend this problem. In fact, in general, (F) is false.

For $(n,k)= (3,0)$, the inclusion (F) states that
\[ \mathrm{Ann}_3(A)*A^{(3)}=1,\]
which fails to hold in general. To exhibit a counterexample, we shall use the following construction of skew braces that is due to the author  \cite[Section 8]{TsangIto}. Here, for simplicity, we assume that $B$ is abelian, which is not needed for the construction. 

\begin{lemma}\label{lem:construction} Let $B = (B,+)$ and $C=(C,+)$ be abelian groups. Let
\begin{align*}
\phi : C\longrightarrow \mathrm{Aut}(B);&\,\ c\mapsto \phi_c\\
\psi: B \longrightarrow\mathrm{Aut}(C);&\,\ b\mapsto \psi_b
\end{align*}
be homomorphisms such that
\begin{equation}\label{phipsicondition}
\mathrm{Im}(\psi_b-\mathrm{id}_C) \subseteq \ker(\phi),
\end{equation}
for all $b\in B$. On the Cartesian product $B\times C$, define
\begin{align*}
    (b,c)\cdot (x,y) & = (b+ \phi_{c}(x),c+y)\\
    (b,c) \circ (x,y) & = (b+x,c+\psi_b(y))
\end{align*}
for all $b,x\in B$ and $c,y\in C$. Then $A = (B\times C,\cdot,\circ)$ is a skew brace, and
\begin{align}\label{BC*}
 (b,c) * (x,y) &= (( \phi_{-c}-\mathrm{id}_B)(x), (\psi_b-\mathrm{id}_C)(y))\\\label{BC[]}
 [(b,c), (x,y)] & = ( (\mathrm{id}_B -\phi_y)(b) + (\phi_c-\mathrm{id}_B)(x), 0)
\end{align}
for all $b,x\in B$ and $c,y\in C$, where $[\,\ ,\,\ ]$ denotes the commutator in $(A,\cdot)$.
\end{lemma}
\begin{proof} That $A = (B\times C,\cdot,\circ)$ is a skew brace follows from \cite[Lemma 8.1]{TsangGrun}. 

For all $b,x\in B$ and $c,y\in C$, we have
\begin{align*}
(b,c)*(x,y) & = (\phi_{-c}(-b),-c)\cdot (b+x,c+\psi_b(y)) \cdot (\phi_{-y}(-x),-y)\\
& = (\phi_{-c}(x),\psi_b(y))\cdot (\phi_{-y}(-x),-y)\\
&= (\phi_{-c}(x) + \phi_{\psi_b(y) -y}(-x), \psi_b(y) -y )\\
& = (( \phi_{-c}-\mathrm{id}_B)(x), (\psi_b-\mathrm{id}_C)(y))
\end{align*}
by the condition (\ref{phipsicondition}), and 
\begin{align*}
[(b,c),(x,y)] & = (b+\phi_c(x),c+y) \cdot (\phi_{-c}(-b),-c)\cdot (\phi_{-y}(-x),-y)\\
& = (b+\phi_c(x) - \phi_y(b),y)\cdot(\phi_{-y}(-x),-y)\\
&= ( (\mathrm{id}_B -\phi_y)(b) + (\phi_c-\mathrm{id}_B)(x), 0).
\end{align*}
This proves the identities (\ref{BC*}) and (\ref{BC[]}).
\end{proof}

\begin{example}Let $p\geq 5$ be any prime. Let us take $B = C= (\mathbb{F}_p^4,+)$. For brevity, let $\vec{e}_1,\vec{e}_2,\vec{e}_3,\vec{e}_4$ denote the standard basis of $\mathbb{F}_p^4$, and identify
\[ \mathrm{Aut}(B) = \mathrm{GL}_4(\mathbb{F}_p) = \mathrm{Aut}(C)\]
via this basis. Consider the homomorphisms defined by
\begin{align*}
    \phi : C \longrightarrow \mathrm{Aut}(B);&\,\ \phi_{\vec{e}_4} = \left[\begin{smallmatrix} 1 &1  && \\ & 1 &1&\\ && 1 & 1 \\ &&&1 \end{smallmatrix}\right],\, \ker(\phi) = \langle \vec{e}_1,\vec{e}_2,\vec{e}_3\rangle,\\
    \psi : B \longrightarrow \mathrm{Aut}(C);&\,\ \psi_{\vec{e}_3} = \left[\begin{smallmatrix} 1 & 1 && \\ & 1 &1&\\ && 1 &  \\ &&&1 \end{smallmatrix}\right],\, \ker(\psi) = \langle\vec{e}_1,\vec{e}_2,\vec{e}_4\rangle.
\end{align*}
For any $c\in C$, we have
\[ \mathrm{Im}(\phi_c - \mathrm{id}_B)\subseteq \langle \vec{e}_1,\vec{e}_2,\vec{e}_3\rangle,\quad \ker(\phi_c - \mathrm{id}_B)  \supseteq \langle\vec{e}_1\rangle,\]
with equalities when $c\in \langle\vec{e}_4\rangle$ but $c\neq \vec{0}$. For any $b\in B$, we similarly have
\[\mathrm{Im}(\psi_b-\mathrm{id}_C)   \subseteq \langle \vec{e}_1,\vec{e}_2\rangle,\quad \ker(\psi_b -\mathrm{id}_C)  \supseteq \langle \vec{e}_1,\vec{e}_4\rangle,\]
with equalities when $b\in \langle\vec{e}_3\rangle$ but $b\neq \vec{0}$. We see that (\ref{phipsicondition}) is satisfied.

On the one hand, it is clear from (\ref{BC*}) and the above that
\begin{align*}
A^{(2)}  = \langle \vec{e}_1,\vec{e}_2,\vec{e}_3\rangle \times \langle \vec{e}_1,\vec{e}_2\rangle,\quad 
A^{(3)} = \langle\vec{0}\rangle \times \langle \vec{e}_1,\vec{e}_2\rangle.
\end{align*}
On the other hand, it is not hard to use (\ref{BC*}) and (\ref{BC[]}) to verify that
\begin{align}\label{3inclusions}
\mathrm{Ann}_1(A) & \supseteq \langle\vec{e}_1\rangle \times \langle \vec{e}_1\rangle,\\\notag
\mathrm{Ann}_2(A) & \supseteq \langle \vec{e}_1,\vec{e}_2\rangle \times \langle \vec{e}_1,\vec{e}_2\rangle,\\\notag
\mathrm{Ann}_3(A)&\supseteq \langle \vec{e}_1,\vec{e}_2,\vec{e}_3\rangle \times \langle\vec{e}_1,\vec{e}_2,\vec{e}_3\rangle.
\end{align}
More specifically, for any $b,c\in \langle \vec{e}_1,\vec{e}_2,\vec{e}_3\rangle$ and $x,y \in \mathbb{F}_p^4$, note that
\begin{align*}
(b,c) * (x,y) & = (\vec{0},(\psi_b-\mathrm{id}_C)(y))\\
(x,y) * (b,c) & = ((\phi_{-y}-\mathrm{id}_B)(b), (\psi_{x}-\mathrm{id}_C)(c))\\
[(b,c),(x,y)] & = ((\mathrm{id}_B-\phi_y)(b),\vec{0})
\end{align*}
because $c\in \ker(\phi)$. By the definition of $\psi$, we have
\[ (\psi_b - \mathrm{id}_C)(y) \in \begin{cases}
 \langle \vec{0}\rangle & \mbox{for }b\in \langle\vec{e}_1,\vec{e}_2\rangle,\\
    \langle \vec{e}_1,\vec{e}_2\rangle &\mbox{for }b\in \langle\vec{e}_1,\vec{e}_2,\vec{e}_3\rangle,
\end{cases}\]
\[ (\psi_x-\mathrm{id}_C)(c) \in
\begin{cases}
 \langle\vec{0}\rangle &\mbox{for }c\in \langle\vec{e}_1\rangle,\\
 \langle\vec{e}_1\rangle &\mbox{for }c\in\langle\vec{e}_1,\vec{e}_2\rangle,\\
  \langle\vec{e}_1,\vec{e}_2\rangle &\mbox{for }c\in\langle\vec{e}_1,\vec{e}_2,\vec{e}_3\rangle.
\end{cases}
\]
By the definition of $\phi$, we similarly have
\[ (\phi_{-y}-\mathrm{id}_B)(b), (\mathrm{id}_B-\phi_y)(b) \in
\begin{cases}
 \langle \vec{0}\rangle & \mbox{for }b\in\langle\vec{e}_1\rangle,\\
 \langle\vec{e}_1\rangle & \mbox{for }b\in \langle\vec{e}_1,\vec{e}_2\rangle,\\
 \langle\vec{e}_1,\vec{e}_2\rangle & \mbox{for }b\in \langle\vec{e}_1,\vec{e}_2,\vec{e}_3\rangle.
\end{cases}
\]
The three inclusions in (\ref{3inclusions}) follow from these observations.

Now, a simple calculation using (\ref{BC*}) yields 
\[ (\vec{e}_3,\vec{0})* (\vec{0},\vec{e}_2) = (\vec{0},(\psi_{\vec{e}_3}-\mathrm{id}_C)(\vec{e}_2))= (\vec{0},\vec{e}_1).\]
This implies that $\mathrm{Ann}_3(A) * A^{(3)} \neq 1$, giving a desired counterexample.
\end{example}

We conclude that the inclusions (A) $\sim$ (H) are all false except for (E).

\section{Another lower central series and annihilator nilpotency}\label{sec:central}

In the previous sections, we mostly discussed the left, right, socle, and annihilator series. But the different types of nilpotency that they define are not compatible. In this section, we shall delve into annihilator nilpotency -- this is the strongest type of nilpotency by Theorem \ref{thm:nilpotency2}.

Although annihilator nilpotency was defined via the annihilator series, an analog of the upper central series, it also admits an equivalent characterization in terms of an analog of the lower central series, as follows.

The \textit{lower central series} of $A$ is defined by
\begin{equation}\label{gamma series}
\Gamma_1(A) = A,\quad \Gamma_{n+1}(A) = \langle \Gamma_{n}(A)*A,\,  A*\Gamma_n(A),\, [A,\Gamma_n(A)] \rangle,
\end{equation}
for all $n\in \mathbb{N}$, where the commutator $[\,\ ,\,\ ]$ and the subgroup generation $\langle \,\ \rangle$ are both taken in $(A,\cdot)$. This was introduced by Bonatto and Jedli\v{c}ka in \cite{central}, but note that the index starts from $0$ in their definition. Before \cite[Lemma 2.6]{central}, they mentioned that they do not know whether the $\Gamma_n(A)$ are ideals or even left ideals of $A$. However, as noted in \cite{nilpotent'} (without proof), for example, they are ideals of $A$. This may be proven in a very similar manner to the proof of Proposition \ref{prop:right}. We spell out the details for the convenience of the reader.

\begin{proposition}\label{prop:gamma'}For all $n\geq 1$, we have that $\Gamma_n(A)$ is an ideal of $A$.
\end{proposition}
\begin{proof}The case $n=1$ is trivial. Now, suppose that $\Gamma_n(A)$ is an ideal of $A$.
\begin{enumerate}[(1)]
\item $\Gamma^{(n+1)}(A)$ is a left ideal of $A$: For any $a,b\in A$ and $x\in \Gamma_n(A)$, the fact that $\Gamma_n(A)$ is an ideal of $A$ implies $a\circ x\circ \overline{a},\, \lambda_a(x)\in \Gamma_n(A)$, and so
\begin{align*} 
\lambda_a(x*b)& = (a\circ x\circ \overline{a})*\lambda_a(b) \in \Gamma_n(A)*A\\
\lambda_a(b*x) & = (a\circ b\circ \overline{a})*\lambda_a(x) \in A*\Gamma_n(A)\\
\lambda_a([b,x]) & = [\lambda_a(b),\lambda_a(x)]\in [A,\Gamma_n(A)]
\end{align*}
by Lemma \ref{lem:identities}(3). Since the elements $x*b,b*x,[b,x]$ generate $\Gamma_{n+1}(A)$ with respect to $\cdot$ and $\lambda_a\in\Aut(A,\cdot)$, this yields $\lambda_a(\Gamma_{n+1}(A))\subseteq \Gamma_{n+1}(A)$, for all $a\in A$. Hence, indeed $\Gamma_{n+1}(A)$ is a left ideal of $A$.
\item $\Gamma_{n+1}(A)$ is normal in $(A,\cdot)$: For any $a,b\in A$ and $x\in \Gamma_n(A)$, we have
\begin{align*}
a\cdot (x*b)\cdot a^{-1} & = (x*a)^{-1}\cdot (x*(a\cdot b)) \in \Gamma_{n}(A)*A
\end{align*}
by Lemma \ref{lem:identities}(1). Now, the hypothesis that $\Gamma_n(A)$ is an ideal of $A$ also implies that $b*x, axa^{-1}\in \Gamma_n(A)$, so we see that
\begin{align*}
a\cdot (b*x) \cdot a^{-1} & = [a,b*x]\cdot (b*x) \in [A,\Gamma_n(A)]\cdot (A*\Gamma_n(A)),\\
a\cdot [b,x]\cdot a^{-1} & = [aba^{-1},axa^{-1}]\in [A,\Gamma_n(A)].
\end{align*}
Since the elements $x*b,b*x,[b,x]$ generate $\Gamma_{n+1}(A)$ with respect to $\cdot$, it then follows that $\Gamma_{n+1}(A)$ is normal in $(A,\cdot)$.
\item $\Gamma_{n+1}(A)$ is normal in $(A,\circ)$: For any $a\in A$ and $y\in \Gamma_{n+1}(A)$, clearly
\[ y(y*\overline{a}) \in \Gamma_{n+1}(A)\Gamma_{n+2}(A)\subseteq \Gamma_{n+1}(A).\]
Since $\Gamma_{n+1}(A)$ is a left ideal of $A$ and is normal in $(A,\cdot)$, we have
\[ a\circ y\circ \overline{a} = a\cdot \lambda_a(y(y*\overline{a}))\cdot a^{-1}\in \Gamma_{n+1}(A)\]
by Lemma \ref{lem:identities}(4), and so $\Gamma_{n+1}(A)$ is normal in $(A,\circ)$.
\end{enumerate}
This completes the proof that $\Gamma_{n+1}(A)$ is an ideal of $A$.
\end{proof}

The next result is from \cite[Theorem 2.7]{central}.

\begin{theorem}\label{thm:AnnGamma}
For all $n\geq1$, we have 
\[ \mathrm{Ann}_{n-1}(A)=A\,\ \iff \,\ \Gamma_n(A)=1.\]
\end{theorem}
\begin{proof}For any $i,j\in \mathbb{N}$, it follows from the definition that
\[ \Gamma_i(A) \subseteq \mathrm{Ann}_j(A) \,\ \iff \,\ \Gamma_{i+1}(A)\subseteq \mathrm{Ann}_{j-1}(A). \]
Applying this equivalence $n-1$ times yields
\[ \Gamma_{1}(A) \subseteq \mathrm{Ann}_{n-1}(A)\,\ \iff \,\ \Gamma_n(A) \subseteq \mathrm{Ann}_0(A) ,\]
which proves the claim.
\end{proof}

Therefore, perhaps \eqref{gamma series} and the annihilator series \eqref{ann series} are, respectively, the ``correct" analogs of the lower and upper central series,  in the setting of skew braces. There is another reason to believe that this is indeed the case. A categorical approach was employed to study skew braces in \cite{SKB}, where it was shown that skew braces form a semi-abelian category, and the so-called \textit{Huq commutator} of ideals is given as follows. Let us remark in passing that the categorical aspects of Hopf braces, an algebraic structure defined in \cite{Hopf} as a Hopf-theoretic generalization of skew braces, were also studied in \cite{GS}.

\begin{definition}\label{def:commutator} For any ideals $I$ and $J$ of $A$, define their \textit{commutator} by
\begin{align*}
 [I,J]^A & = \langle [I,J],\, [I,J]_\circ,\, I*J\rangle^A \\
& =  \langle [x,y],\, [x,y]_\circ,\, x*y: x\in I, \, y\in J
\rangle^A,
\end{align*}
where $[\,\ ,\,\ ]$ and $[\,\ ,\,\ ]_\circ$ denote the additive and multiplicative commutators, respectively, and the angle brackets $\langle \,\ \rangle^A$ denote ideal generation in $A$. 
\end{definition}

Before proceeding, let us make a few remarks about this commutator.

\begin{remark}\label{remark} Let $I$ and $J$ be ideals of $A$.
\begin{enumerate}[(1)]
\item The commutator of $I$ and $J$ was originally described as
\[ [I,J]^A = \langle[x,y],\, [x,y]_\circ,\, (x\circ y)\cdot (x\cdot y)^{-1}\rangle^A\]
in the statement of \cite[Proposition 3.18]{SKB}. But
\[(x\circ y)\cdot (x\cdot y)^{-1} = x\cdot (x*y)\cdot x^{-1},\]
so clearly Definition \ref{def:commutator} is an equivalent definition.
\item The commutator of $I$ and $J$ can also be equivalently defined as
\begin{align*}
[I,J]^A & = \langle [I,J],\, I*J,\, J*I\rangle^A\\
& = \langle [x,y],\, x*y,\, y*x: x\in I,\, y\in J\rangle^A.
\end{align*}
This is because the congruence equations
\begin{align*}
 y\circ x&\equiv x\circ y \equiv x\cdot y \equiv y\cdot x\pmod{\langle [I,J]_\circ,\, I*J,\, [I,J]\rangle^A}\\
  y\circ x&\equiv y\cdot x\equiv x\cdot y \equiv x\circ y\pmod{\langle J*I,\, [I,J],\, I*J\rangle^A}
\end{align*}
hold for all $x\in I$ and $y\in J$.
\item The commutator of $I$ and $J$ is symmetric, namely
\[[I,J]^A = [J,I]^A.\]
This is clear from the observation (2) above.
\end{enumerate}
\end{remark} 

Unlike the star product, it does not matter whether we take the commutator from the left or from the right. Hence, using this commutator instead of the star product, there is a unique natural way to define an analog of the lower central series, namely we define
\begin{equation}\label{gamma' series} \Gamma_1'(A) = A,\quad \Gamma_{n+1}'(A) = [A,\Gamma_n'(A)]^A,\end{equation}
for all $n\in\mathbb{N}$. This coincides with the series \eqref{gamma series}. We note that Remark \ref{remark} and Proposition~\ref{prop:gamma} were observed in \cite{central-ideal} as well.

\begin{proposition}\label{prop:gamma}For all $n\geq1$, we have $\Gamma_n(A) = \Gamma_n'(A)$.
\end{proposition}
\begin{proof} This follows from Remark \ref{remark}(2) and Proposition \ref{prop:gamma'}.
\end{proof}

Therefore, the series \eqref{gamma series} can be defined using a commutator operator in the same way the lower central series of a group is defined. Our discussion so far also suggests that annihilator nilpotency is perhaps the more appropriate analog of nilpotency in the setting of skew braces.

In the case of finite groups, a famous theorem of Fitting (see \cite[(5.2.8)]{Robinson} for example) states that the product of finitely many nilpotent normal subgroups is again nilpotent. This fact is required to define the Fitting subgroup of a finite group.

In the case of finite skew braces, the analog of Fitting's Theorem is false by \cite[Example~B]{central-ideal} -- there is a skew brace of order $32$ that is not annihilator nilpotent but is expressible as the product of two ideals that are annihilator nilpotent as skew braces. To resolve this issue, the authors of \cite{central-ideal} introduced the notion of ``relative" annihilator nilpotency for ideals, as follows.

Let $I$ be an ideal of $A$. According to \eqref{gamma' series} and Proposition \ref{prop:gamma}, as a skew brace, the lower central series of $I$ is defined by
\[ \Gamma_1(I) =I,\quad \Gamma_{n+1}(I) = [I,\Gamma_n(I)]^I,\]
for all $n\in\mathbb{N}$. According to \cite[Section 5]{central-ideal}, as an ideal of $A$, the \textit{lower central series of $I$ with respect to $A$} is defined by
\[ \Gamma_1(I)^A=I,\quad \Gamma_{n+1}(I)^A= [I,\Gamma_n(I)^A]^A, \]
for all $n\in\mathbb{N}$; the superscript $A$ on $\Gamma_n(I)$ is missing in \cite{central-ideal} and it was a typo. Note that this is not the same lower central series in \cite[Definition 2.3]{central}.

\begin{definition}An ideal $I$ of $A$ is called  \textit{annihilator nilpotent with res\-pect to $A$} if $\Gamma_n(I)^A=1$ for some $n$. 
\end{definition}

Let $I$ be an ideal of $A$. If $I$ is annihilator nilpotent with respect to $A$, then $I$ is an annihilator nilpotent skew brace, but not vice versa. One can also see this by considering upper central series. Indeed, it is not hard to see that $I$ is annihilator nilpotent if and only if there exists a finite chain
\[ 1 = J_0 \leq J_1\leq \cdots \leq J_{m-1} \leq J_m = I\]
of ideals of $I$ such that $J_i/J_{i-1}\leq\mathrm{Ann}(I/J_{i-1})$ for all $1\leq i\leq m$. As noted in \cite[Definition~5.1]{central-ideal}, we similarly have that $I$ is annihilator nilpotent with respect to $A$ if and only if there exists a finite chain
\[ 1 = J_0 \leq J_1\leq \cdots \leq J_{m-1} \leq J_m = I\]
of ideals of $A$ such that $J_i/J_{i-1}\leq\mathrm{Ann}(I/J_{i-1})$ for all $1\leq i\leq m$; note that it says ``ideals of $I$" in \cite{central-ideal} and it was a typo. By strengthening the notion of annihilator nilpotency for ideals this way, the following result was obtained in \cite[Theorem 5.3]{central-ideal}.

\begin{theorem}\label{thm:fitting} Let $I$ and $J$ be any ideals of $A$. If $I$ and $J$ are annihilator nilpotent with respect to $A$, then $IJ$ is also annihilator nilpotent with respect to $A$. In fact, if $\Gamma_{m}(I)^A=1$ and $\Gamma_{n}(J)^A=1$, then $\Gamma_{m+n-1}(IJ)^A=1$.
\end{theorem}

As introduced in \cite{central-ideal}, the \textit{Fitting ideal} of $A$ is the ideal $\mathrm{Fit}(A)$ generated by the ideals of $A$ that are annihilator nilpotent with respect to $A$. Thanks to Theorem \ref{thm:fitting}, we know that $\mathrm{Fit}(A)$ is annihilator nilpotent with respect to $A$ when $A$ is finite. The authors of \cite{central-ideal} also introduced \textit{Frattini ideal} and proved an analog of a celebrated result of Gasch\"{u}tz that relates the Frattini and Fitting subgroups (see \cite[Theorem 5.10]{central-ideal}).

\subsection*{Acknowledgements}
This research is supported by JSPS KAKENHI Grant Number 24K16891. The author recognizes that Theorem \ref{main thm}, which is a natural extension of her previous work \cite[Proposition 2.2]{TsangGrun}, arose out of discussion with her student Risa Arai. The author would also like to thank the referees for their helpful comments and suggestions, which helped improve the paper significantly.





\EditInfo{March 4, 2025}{July 27, 2025}{David Towers and Ivan Kaygorodov}

\end{document}